\documentclass[10pt]{article}
\usepackage[utf8]{inputenc}

\usepackage{amssymb,amsthm,amsmath,mathrsfs}
\usepackage{enumerate}
\usepackage{graphicx,color}
\usepackage[hidelinks]{hyperref}
\usepackage{tikz-cd}

\newcommand{\dd}{\mathrm{d}}
\newcommand{\E}{\mathbb{E}}

\newcommand{\R}{\mathbb{R}}
\newcommand{\C}{\mathbb{C}} 
\newcommand{\Z}{\mathbb{Z}}
\newcommand{\e}{\varepsilon}
\newcommand{\sL}{\mathscr{L}}
\newcommand{\p}[1]{\mathbb{P}\left( #1 \right)}
\newcommand{\scal}[2]{\!\left\langle #1, #2 \right\rangle\!}

\DeclareMathOperator{\Var}{Var}
\DeclareMathOperator{\sgn}{sgn}

\newtheorem{theorem}{Theorem}
\newtheorem{lemma}[theorem]{Lemma}
\newtheorem{corollary}[theorem]{Corollary}

\theoremstyle{remark}
\newtheorem{remark}[theorem]{Remark}

\theoremstyle{definition}

\title{\vspace{-3em}
Khinchin-type inequalities via Hadamard's factorisation
}

\author{Alex Havrilla\thanks{Georgia Institute of Technology; Atlanta, GA 30332. Research supported in part by NSF grant DMS-1955175.}, \ Piotr Nayar\thanks{\linespread{1.0} University of Warsaw, Banacha 2, 02-097 Warszawa, Poland. Research supported in part by the National Science Centre, Poland, grant 2018/31/D/ST1/01355.}
\ and 
Tomasz Tkocz\thanks{\linespread{1.0} Carnegie Mellon University; Pittsburgh, PA 15213, USA. Email: ttkocz@math.cmu.edu. Research supported in part by NSF grant DMS-1955175.}
\date{7th October 2021}
}

\begin{document}

\maketitle

\begin{abstract}
We prove Khinchin-type inequalities with sharp constants for type L random variables and all even moments. Our main tool is Hadamard's factorisation theorem from complex analysis, combined with Newton's inequalities for elementary symmetric functions. Besides the case of independent summands, we also treat ferromagnetic dependencies in a nonnegative external magnetic field (thanks to Newman's generalisation of the Lee-Yang theorem). Lastly, we compare the notions of type L, ultra sub-Gaussianity (introduced by Nayar and Oleszkiewicz) and strong log-concavity (introduced by Gurvits), with the latter two being equivalent.
\end{abstract}

\bigskip

\begin{footnotesize}

\end{footnotesize}

\bigskip

\section{Introduction and results}

Motivated by Ising models of Lee-Yang type, Newman in \cite{N1} made the following definition: a (real-valued) random variable $X$ is of \textbf{type $\mathscr{L}$}, if (i) for some positive constants $C$ and $C'$, $|\E\exp(zX)| \leq C\exp(C'|z|^2)$ for all $z \in \C$ and (ii) $\E\exp(zX)$ is even with only pure imaginary zeros. Note that (ii) implies that $X$ is symmetric, that is $-X$ has the same distribution as $X$. Perhaps the simplest nontrivial example is a Rademacher random variable (a symmetric random sign) taking the values $\pm 1$ with probability $\frac{1}{2}$.
Newman proved that if $X$ is of type $\mathscr{L}$, then for an even integer $p \geq 2$,
\[
\|X\|_p \leq \|G\|_p\|X\|_2,
\]
where $G$ is a standard Gaussian random variable (see \cite{N2} or Theorem 5 in \cite{N1}). As usual, $\|Y\|_p = (\E |Y|^p)^{1/p}$ denotes the $p$-norm of a random variable $Y$.

The goal of this short note is to extend this result to a comparison inequality for moments of arbitrary even orders. As explained later, the assumption of evenness in (ii) may be easily dropped (it is just a matter of cosmetic shifting). Thus we shall say that a random variable $X$ is of \textbf{type $\mathscr{L}'$} if it satisfies (i) and (ii') $\E\exp(zX)$ has only pure imaginary zeros.
Our main result is the following theorem.

\begin{theorem}\label{thm:mom-comp}
Let $X$ be a random variable of type $\mathscr{L}'$. Then for every even integers $2 \leq p \leq q$, we have
\begin{equation}\label{eq:mom-comp}
\|X\|_q \leq \frac{\|G\|_q}{\|G\|_p}\|X\|_p,
\end{equation}
where $G$ is a standard Gaussian random variable.
\end{theorem}

The usefulness of this result in Khinchin inequalities for sums of independent random variables comes from the following evident fact.

\begin{remark}\label{rem:typeL-sums}
Finite sums of independent type $\mathscr{L}$ (resp. $\sL'$) random variables are of type $\mathscr{L}$ (resp. $\sL'$).
\end{remark}

Thus the above theorem immediately gives classical Khinchin inequalities with sharp constants for even moments, the result due to Nayar and Oleszkiewicz from \cite{NO}. Furthermore, it automatically extends to Hilbert space coefficients by isometrical embeddings into $L_p$ spaces based on standard Gaussians (Orlicz-Szarek's argument, see Remark 3 in \cite{Sz}).   

\begin{corollary}\label{cor:Hilbert}
Let $(H,\|\cdot\|)$ be a separable (real or complex) Hilbert space. If $X_1, \ldots, X_n$ are independent type $\sL'$ random variables, then for every vectors $v_1, \ldots, v_n$ in $H$, the sum $X = \sum_{j=1}^n X_jv_j$ satisfies \eqref{eq:mom-comp} for all positive even integers $p \leq q$, where we denote $\|X\|_p = (\E\|X\|^p)^{1/p}$.
\end{corollary}

Thanks to Newman's generalisation of the Lee-Yang theorem and Griffiths' method of a ``ghost'' spin (see \cite{G, N0, N3}), we moreover obtain \eqref{eq:mom-comp} when $X$ is a positive linear combination of random variables with ferromagnetic dependencies $J=(J_{jk})$ in a nonnegative external magnetic field $h=(h_j)$.

\begin{corollary}\label{cor:ferr}
Let $\mu_1, \ldots, \mu_n$ be Borel probability measures on $\R$, each one of type $\sL$. Suppose that $(X_1,\dots,X_n)$ is a random vector in $\R^n$ whose law $\rho$ on $\R^n$ is of the form
\begin{equation}\label{eq:ferr-density}
\dd\rho(x_1,\ldots,x_n) = Z^{-1}\exp\left(\sum_{j=1}^n h_jx_j+\sum_{j,k=1}^n J_{jk}x_jx_k\right)\dd\mu_1(x_1)\dots\dd\mu_n(x_n)
\end{equation}
with $h_j \geq 0$, $J_{jk} \geq 0$ for all $j,k \leq n$, where $Z$ is the normalising constant. Then for every nonnegative $a_1, \ldots, a_n$, the sum $X = \sum_{j=1}^n a_jX_j$ satisfies \eqref{eq:mom-comp} for all positive even integers $p \leq q$.
\end{corollary}

For a certain subclass of the type $\sL$ random variables (cf. the section below devoted to examples), we are also able to derive sharp moment comparison between the second and $p$-th moment, $p \geq 3$.

\begin{theorem}\label{thm:p>3}
Let $X$ be type $\sL$ random variable with characteristic function of the form $\phi_X(t) = e^{-a t^2/2}\prod_{j=1}^\infty(1-b_jt^2)$ with $a > 0$, $b_j \geq 0$, $\sum b_j \leq a$. Let $\sigma = \sqrt{\Var(X)}$. Then for every $p \geq 3$, 
\[
\E|\sigma Z_1|^p \leq \E|X|^p \leq \E|\sigma Z_0|^p,\]
where $Z_0$ is a standard Gaussian random variable and $Z_1$ is a random variable with density $(2\pi)^{-1/2}x^2e^{-x^2/2}$.
\end{theorem}

First we prove Theorem \ref{thm:mom-comp} (relying on the product form of the moment generating function $\E\exp(zX)$ and exploiting Newton's inequalities for the elementary symmetric functions) and show how to obtain Corollaries \ref{cor:Hilbert} and \ref{cor:ferr}. Then we provide several examples of type $\mathscr{L}$ random variables and prove Theorem \ref{thm:p>3} (using techniques from \cite{ENT2, NZ}, relying on convexity properties of certain functions related to $|\cdot|^p$, hence the restriction $p \geq 3$). We conclude with a discussion comparing type $\mathscr{L}$, ultra sub-Gaussianity (introduced in \cite{NO}) and strong log-concavity (introduced in \cite{Gur}).

\subsection*{Acknowledgments} 
We should very much like to thank the referees for their careful reading of our manuscript and for their useful comments.

\section{Even moments}

We recall that a nonnegative sequence $(a_n)_{n \geq 0}$ is called log-concave if it is supported on a contiguous interval (which may be infinite) and satisfies $a_n^2 \geq a_{n-1}a_{n+1}$, $n \geq 1$. We remark that log-concavity implies the inequalities $a_ka_l \geq a_{k+j}a_{l-j}$ for all $k \geq l$, $0 \leq j \leq l$. We refer for instance to the classical survey \cite{stan} by Stanley.

\subsection{Proof of Theorem \ref{thm:mom-comp} for Rademacher sums}

It is instructive to first consider the case when $X$ from Theorem \ref{thm:mom-comp} is of the form $\sum_{j=1}^n a_j\e_j$, where $a_1, \dots, a_n$ are real coefficients and $\e_1, \e_2, \dots$ are i.i.d. Rademacher random variables (symmetric random signs, that is for each $j$, $\p{\e_j = -1} = \p{\e_j = 1} = \frac{1}{2}$). The proof in this case already encapsulates almost all the key ideas needed for the general case and at the same time provides a new, short, elementary and self-contained proof of the main result of \cite{NO}, the aforementioned classical Khinchin inequalities with sharp constants for even moments.

Let $X = \sum_{j=1}^m a_j\e_j$ and for $k \geq 0$, define
\begin{equation}\label{eq:def-rn}
r_k = \frac{\E X^{2k}}{\E G^{2k}}
\end{equation}
(so $r_0 = 1$ and $r_k = \frac{\E X^{2k}}{(2k)!}2^kk!$). Note that it suffices to show that the sequence $(r_k)_{k=0}^\infty$ is log-concave. This is because then in particular, by monotonicity of slopes, the sequence $\frac{\log r_k-\log r_0}{k}$ is nonincreasing and in view of $r_0 = 1$, this gives that the sequence $(r_k^{1/k})_{k\geq 0}$ is nonincreasing, which is exactly \eqref{eq:mom-comp}.
 
Plainly, since $X$ is symmetric and thanks to independence, for $t\geq 0$,
\[
\sum_{k=0}^\infty \frac{\E X^{2k}}{(2k)!}2^kt^k = \E\exp(\sqrt{2t}X) = \prod_{j=1}^m \E\exp(\sqrt{2t}a_j\e_j) = \prod_{j=1}^m \cosh(\sqrt{2t}a_j).
\]
The main point is that since $\cosh z = \prod_{n=1}^\infty \left(1 + \frac{z^2}{(n-\frac{1}{2})^2\pi^2}\right)$, $z \in \C$, the product on the right hand side takes the form
\[
\prod_{l=1}^m\prod_{n=1}^\infty \left(1 + \frac{2a_l^2}{(n-\frac12)^2\pi^2}t\right) = \prod_{j=1}^\infty (1+\alpha_j t) = \sum_{k=0}^\infty \sigma_kt^k,
\]
with nonnegative $\alpha_j$, where 
\[
\sigma_k = \sigma_k(\{\alpha_j\}) = \sum_{\substack{A \subset \{\alpha_j\} \\ |A| = k}}\prod_{a \in A} a
\]
is the $k$th elementary symmetric function of the (multi)set $\{\alpha_j\}$. Comparing the coefficients thus gives
\[
r_k = \frac{\E X^{2k}}{(2k)!}2^kk! = \sigma_kk!.
\]
Thanks to this expression for $r_k$, Newton's inequalities allow to finish the argument.

\textbf{Claim.} The sequence $(\sigma_kk!)_{k \geq 0}$ is log-concave.
\begin{proof}
Fix $k \geq 1$. By Newton's inequalities,
\[
\frac{\sigma_k(\{\alpha_j\}_{j=1}^n)^2}{\binom{n}{k}^2} \geq \frac{\sigma_{k-1}(\{\alpha_j\}_{j=1}^n)}{\binom{n}{k-1}}\frac{\sigma_{k+1}(\{\alpha_j\}_{j=1}^n)}{\binom{n}{k+1}}.
\]
Letting $n\to\infty$ gives the claim.
\end{proof}

\subsection{Proof of Theorem \ref{thm:mom-comp}}

We proceed with the proof of Theorem \ref{thm:mom-comp} in the general case when $X$ is type $\mathscr{L}'$. First we argue that it suffices to assume that $X$ is of type $\mathscr{L}$ and then give the proof in this case.

\emph{Type $\mathscr{L}$ vs. $\mathscr{L}'$.} We begin with a straightforward observation.

\begin{lemma}\label{lm:LvsL'}
If $X$ is a random variable of type $\sL'$, then there is $c \in \R$ such that $X-c$ is of type $\sL$.
\end{lemma}
\begin{proof}
As in Proposition 2 in \cite{N1}, by the Hadamard factorization theorem,
\begin{equation}\label{eq:rep}
\E\exp(zX) = \exp(bz^2/2+cz)\prod_j (1+\alpha_jz^2/2)
\end{equation}
for some $c \in \R$, $b \geq 0$ and $0< \alpha_1 \leq \alpha_2 \leq \ldots$ with $\sum \alpha_j < \infty$ (the set $\{\alpha_j\}$ may be empty, finite or infinite countable). Thus $Y = X-c$ is of type $\mathscr{L}$.
\end{proof}

\emph{Reduction to type $\mathscr{L}$.}
Let $X$ be a random variable of type $\mathscr{L}'$. Thus $Y = X-c$ is of type $\mathscr{L}$. Since $Y$ is a symmetric random variable,
\begin{equation}\label{eq:mom-same}
\E|X|^p = \E|Y+c|^p = \E|Y+c\varepsilon|^p
\end{equation}
where $\e$ is a Rademacher random variable independent of $Y$. Note that $Y + c\e$ is of type $\mathscr{L}$ (as a sum of two independent random variables of type $\mathscr{L}$). Assuming the theorem holds for all random variables of type $\mathscr{L}$, we thus conclude it for $X$ because $X$ has the same moments as $Y + c\e$.

\emph{Proof for type $\mathscr{L}$.}
Let $X$ be a random variable of type $\mathscr{L}$. Let $r_k$ be given by \eqref{eq:def-rn} and as in the Rademacher case, it suffices to show that the sequence $(r_k)_{k=0}^\infty$ is log-concave. Note that \eqref{eq:rep} holds with $c=0$, so applying it to $z = \sqrt{2t}$ with $t \geq 0$, we obtain
\[
\sum_{k=0}^\infty \frac{r_k}{k!}t^k = \sum_{k=0}^\infty \frac{\E X^{2k}}{(2k)!}2^kt^k = \E\exp(\sqrt{2t}X) = \exp(bt)\prod_j (1+\alpha_jt).
\]
Moreover, as in the Rademacher case, $\prod_j (1+\alpha_j t) = \sum_{k=0}^\infty \sigma_kt^k$, where $\sigma_k = \sigma_k(\{\alpha_j\})$ is the $k$th elementary symmetric function. Expanding $\exp(bt)$, the right hand side then becomes 
\[
\left(\sum_{k=0}^\infty \frac{b^k}{k!}t^k\right)\left(\sum_{k=0}^\infty \sigma_kt^k\right) = \sum_{k=0}^\infty\left(\sum_{j=0}^k \frac{b^{k-j}}{(k-j)!}\sigma_{j} \right) t^k.
\]
By comparing the coefficients,
\[
r_k = k!\sum_{j=0}^k \frac{b^{k-j}}{(k-j)!}\sigma_{j} = \sum_{j=0}^k \binom{k}{j} b^{k-j}\sigma_{j}j!.
\]
The sequence $(b^k)_{k=0}^\infty$ is trivially log-concave and, as noted in the Rademacher case, the sequence $(\sigma_kk!)_{k=0}^\infty$ is also log-concave.  Walkup's theorem (see, e.g. \cite{Gur2, Lig, NO, W}) which says that the binomial convolution preserves log-concavity yields that the sequence $(r_k)_{k \geq 0}$ is also log-concave and the proof of Theorem~\ref{thm:mom-comp} is complete.\hfill$\square$

Incidentally, since $(r_n)_{n \geq 0}$ being log-concave means that $X$ is ultra sub-Gaussian (see Section \ref{sec:notions} and \cite{NO}), we have obtained the following result.

\begin{theorem}\label{thm:type-L-USG}
If a random variable $X$ is of type $\sL'$, then the sequence $(r_n)_{n=0}^\infty$ defined in \eqref{eq:def-rn} is log-concave. In particular, type $\sL$ random variables are ultra sub-Gaussian.
\end{theorem}
\begin{proof}
It has been shown above that if $X$ is of type $\sL$, then $X$ is ultra sub-Gaussian. Consequently, by Lemma \ref{lm:LvsL'} and \eqref{eq:mom-same}, the same holds if $X$ is of type~$\sL'$.
\end{proof}

\subsection{Proof of Corollary \ref{cor:Hilbert}}

First suppose that $H$ is real Hilbert space. Fix an orthonormal basis $e_1, e_2, \ldots$ in $H$. Let $g_1, g_2, \ldots$ be i.i.d. standard Gaussian random variables (independent of the $X_j$). Consider the map
\[
H \ni x \mapsto c_q\sum_{k} \scal{x}{e_k}g_k \in L_q(\Omega,\mathbb{P}).
\]
Since sums of independent Gaussians are Gaussian, this map with $c_q = \|G\|_q^{-1}$ is an isometrical embedding. Given vectors $v_1,\ldots, v_n$ in $H$, for appropriate scalars $a_j$, we thus have 
\begin{equation}\label{eq:isom}
\left\|\sum_{j} X_jv_j\right\|^q = \E_g\left|\sum_j a_jX_j\right|^q, \qquad  a_j = c_q\scal{v_j}{\sum_{k}g_ke_k}.
\end{equation}
Using this twice, Theorem \ref{thm:mom-comp} applied to $\sum_j a_jX_j$ and Minkowski's inequality, we arrive at
\begin{align*}
\E_X\left\|\sum_{j} X_jv_j\right\|^q &= \E_g\E_X\left|\sum_j a_jX_j\right|^q \\
&\leq \left(\frac{\|G\|_q}{\|G\|_p}\right)^{q}\E_g\left(\E_X\left|\sum_j a_jX_j\right|^p\right)^{q/p} \\
&\leq \left(\frac{\|G\|_q}{\|G\|_p}\right)^{q}\left(\E_X\left(\E_g\left|\sum_j a_jX_j\right|^q\right)^{p/q}\right)^{q/p} \\
&= \left(\frac{\|G\|_q}{\|G\|_p}\right)^{q}\left(\E_X\left\|\sum_j X_jv_j\right\|^p\right)^{q/p}.
\end{align*}
This finishes the proof when $H$ is real. For the complex case, the only required modification of the above argument is in the definition of the isometrical embedding which now reads
\[
H \ni x \mapsto c_q\sum_{k} \Big(\text{Re}\scal{x}{e_k}g_k +\text{Im}\scal{x}{e_k}g_k'\Big)  \in L_q(\Omega,\mathbb{P}),
\]
where $g_1, g_1', g_2, g_2', \ldots$ are i.i.d. standard Gaussian random variables. Identity \eqref{eq:isom} holds with $a_j = c_q\Big(\text{Re}\scal{v_j}{\sum_{k}g_ke_k} + \text{Im}\scal{v_j}{\sum_{k}g_k'e_k}\Big)$ and the rest of the argument is unchanged.
\hfill$\square$

\subsection{Proof of Corollary \ref{cor:ferr}}

Let $(X_1, \ldots, X_n)$ be a random vector with distribution given by \eqref{eq:ferr-density} and let $\e$ be an independent Rademacher random variable. By Proposition 6 from \cite{N3}, the vector $(Y_0,Y_1,\dots,Y_n) = (\e,\e X_1,\dots,\e X_n)$ has distribution $\rho'$ of the form
\[
\dd\rho'(x_0,x_1,\ldots,x_n) = Z'^{-1}\exp\left(\sum_{j,k=0}^n J_{jk}'x_jx_k\right)\dd\mu_0(x_0)\dd\mu_1(x_1)\dots\dd\mu_n(x_n),
\]
where $\mu_0$ is the distribution of $\e$, $J_{0,0}' = 0$, $J_{0,k}' = J_{k,0}'= h_k/2$, $J_{jk}' = J_{jk}$, $j,k \geq 1$, so of the form \eqref{eq:ferr-density} with $h \equiv 0$. Therefore, by Theorem 2 from \cite{N2}, for every $a_0, a_1,\dots, a_n \geq 0$, the sum $S = \sum_{j=0}^n a_jY_j = a_0\e + \sum_{j=1}^n a_j\e X_j$ is of type $\sL$ and in particular, $S$ satisfies \eqref{eq:mom-comp}. Hence, taking $a_0 = 0$ yields that $\sum_{j=1}^n a_jX_j$ also satisfies \eqref{eq:mom-comp}.\hfill$\square$

\section{Examples of type $\mathscr{L}$ random variables}\label{sec:examples}

We list some examples of probability distributions of type $\sL$. This is based on several well-known sufficient conditions for Fourier transforms of positive functions to have only real zeros. We mention merely in passing that the development of such conditions, related to the study of the Laguerre-P\'olya class, was initiated mainly with connections to the Riemann hypothesis and refer to the comprehensive survey \cite{sur}. In what follows, $X$ is a symmetric random variable.

\begin{enumerate}[(a)]

\item 
Let $X$ be integer-valued with $\p{X = 0} = p_0$ and $\p{X = -k} = \p{X = k} = p_k$, $k = 1, \dots, n$ for nonnegative $p_0, \dots, p_n$ with $p_0 + 2\sum_{k=1}^n p_k = 1$. 

If $\frac{1}{2}p_0 \leq p_1 \leq \dots \leq p_n$, then $\E \cos(zX) = p_0 + \sum_{k=1}^n (2p_k)\cos(kz)$ has only real zeros, as it follows from the Enestr\"om-Kakeya theorem (see, e.g. Problem III.204 in \cite{PS1}). As a result, $X$ is of type $\sL$. In particular, if $X$ is uniform on $\{-n,\dots, 1, 1, \dots, n\}$ with a possible atom at $0$ satisfying $\p{X = 0} \leq \frac{1}{n+1}$, then $X$ is of type $\sL$ (the last class was also recently studied in \cite{HT} in the context of Khinchin inequalities).

By the symmetry of $X$, the polynomial $Q(w) = \E w^{X+n}$ is self-inversive (the sequence of its coefficients is a palindrome, in other words, $w^{2n}Q(1/w) = Q(w)$). In particular, all its roots are symmetric with respect to the unit circle, that is if $w_0$ is a root of $Q$, then so is $1/w_0$. For instance, if for some $\alpha \geq 1$,
\[
\frac{1}{2}p_0^\alpha + \sum_{k=1}^{n-1} p_k^\alpha \leq \left(\frac{2}{n-2}\right)^{\alpha-1}p_n^\alpha,
\]
where $n$ is the number of nonzero coefficients of $Q$, then $Q$ has zeros only on the unit circle, so $X$ is of type $\sL$, which follows from the main result of \cite{SV}. In general, the Jury test provides an efficient algorithm to determine whether $Q$ has zeros only on the unit circle, given its coefficients, see \cite{J}.

\item
Let $X$ take values in $[-1,1]$ and have a density $f$ (which is even). Each of the following conditions implies that $X$ is of type $\sL$.
\begin{enumerate}[(i)]
\item $f$ is  nondecreasing on $(0,1)$.

\item $f$ is $C^2$ with $f' < 0$ and $f'' < 0$ on $(0,1)$.

\item $f(t) = h(t)^{\alpha}$, where $\alpha > -1$ and $h$ is an entire even function which is real-valued on the real line, $h(1) = 0$, $h(0) > 0$ and $h'(iz)$ is  in the Laguerre-P\'olya class. In particular, $f(t) = \text{const}\cdot (1-t^{2m})^{\alpha}$ for a nonnegative integer $m$.  
\end{enumerate}

Moreover, if $X$ has a density on $\R$ of the following form, then it is of type $\sL$.

\begin{enumerate}
\item[(iv)] $f(t) = \text{const}\cdot e^{-t^{2m}}$ for a nonnegative integer $m$.  

\item[(v)] $f(t) = (2\pi)^{-1/2}e^{-t^2/2}(1-b+ bt^2)$, $0 \leq b \leq 1$.
\end{enumerate}

Condition (i) is justified again by the Enestr\"om-Kakeya theorem combined with a limit argument (see, e.g. Problem III.205 in \cite{PS1}), (ii) is due to P\'olya (see, e.g. Problem V.173 in \cite{PS}), (iii) is due to Ilieff (see \cite{Il}), which gives (iv) by a limit argument (see also Problem V.170 and 171 in \cite{PS}), (v) is justified by a direct computation of the moment generating function which is $(1+bz^2)e^{z^2/2}$. Moreover, if the density of $X$ is of the form $f(t) = \text{const}\cdot e^{-|t|^{\alpha}}$ with $\alpha \geq 2$, $\alpha \notin \{2,4,\dots\}$, then its characteristic function has infinitely many non-real zeros, in particular $X$ is not of type $\sL$ (see the solution of Problem V.171 in \cite{PS}).
\end{enumerate}

We refer to \cite{N1} for additional examples important in statistical mechanics.

\section{Type $\sL$ random variables with ``enough Gaussianity''}

The goal in this section is to prove Theorem \ref{thm:p>3}. We need a little bit of preparation. For a real parameter $b$, consider the following function
\[
\phi_b(t) = e^{-t^2/2}(1-bt^2), \qquad t \in \R,
\]
which is in $L_1$. By the Fourier inversion formula, we have,
\[
f_b(x) = \frac{1}{2\pi}\int_{-\infty}^\infty \phi_b(t)e^{-itx}\dd t = (1-b+bx^2)\frac{e^{-x^2/2}}{\sqrt{2\pi}}, \qquad x \in \R.
\]
Clearly, $f_b$ is nonnegative on $\R$ if and only if $b \in [0,1]$. Thus $\phi_b$ is the characteristic function of a random variable precisely for every $b \in [0,1]$ and then its density is $f_b$ and variance is the coefficient $2\phi_b(it)[t^2] = 1+2b$. Throughout, $Z_b$ is a random variable with characteristic function $\phi_b$. In particular, $Z_0$ is standard Gaussian and $Z_1$ has the bimodal density $(2\pi)^{-1/2}x^2e^{-x^2/2}$.

\begin{lemma}\label{lm:ch.f.}
For every $a > 0$ and $b_1, b_2, \ldots \geq 0$ with $\sum b_j \leq a$, the function
\[
e^{-at^2/2}\prod_j (1-b_jt^2)
\]
is the characteristic function of a type $\sL$ random variable.
\end{lemma}
\begin{proof}
Let $X_1, X_2, \ldots$ be i.i.d. copies of $Z_1$ and let $Z_0$ be an independent standard Gaussian. Then $\sum \sqrt{b_j}X_j + \sqrt{a-\sum b_j}Z_0$ (which converges in $L_2$ if infinitely many $b_j$ are positive) is the desired random variable. 
\end{proof}

\begin{lemma}\label{lm:interlacing}
For $\lambda \in (0,1)$, let $g_\lambda$ be the density of $\sqrt{\lambda}X_1 + \sqrt{1-\lambda}X_2$, where $X_1, X_2$ are independent copies of $Z_1$. Then for every $0 < \lambda_1 < \lambda_2 < \frac{1}{2}$, the function $g_{\lambda_2} - g_{\lambda_1}$ on $(0,+\infty)$ has exactly two zeros and the sign pattern $+-+$.
\end{lemma}
\begin{proof}
By a direct computation, 
\[g_\lambda(x) = \Big(x^2+\lambda(1-\lambda)(3-6x^2+x^4)\Big)\frac{e^{-x^2/2}}{\sqrt{2\pi}}
\]
so $g_{\lambda_2} - g_{\lambda_1}$ has the same sign as $(\lambda_2-\lambda_1)(1-\lambda_1-\lambda_2)(3-6x^2+x^4)$.
\end{proof}

\begin{lemma}\label{lm:schur}
Let $X_1, X_2, \dots$ be i.i.d. copies of $Z_1$ and let $Y$ be a symmetric random variable independent of the $X_j$. Then the function
\[
\Psi(b_1,\dots,b_n) = \E|\sqrt{b_1}X_1+\dots+\sqrt{b_n}X_n + Y|^p
\]
is Schur-concave on $[0+\infty)^n$.
\end{lemma}
\begin{proof}
We use the technique of interlacing densities (see, e.g. \cite{ENT2} or \cite{NZ}). Let $h(x) = |x+1|^p+|x-1|^p$. It suffices to show that for every $0 < \lambda_1 < \lambda_2 < \frac{1}{2}$, we have
\[
\int_0^\infty h(x)(g_{\lambda_2(x)} - g_{\lambda_1}(x)) \dd x \geq 0,
\]
where $g_\lambda$ is as in Lemma \ref{lm:interlacing}. For arbitrary $\alpha, \beta$, 
\[
\int (\alpha x^2+\beta)(g_{\lambda_2}(x) - g_{\lambda_1}(x)) \dd x = 0.\]
This is because $g_\lambda$ is a probability density, so $\int g_\lambda(x) \dd x = 1$ and by its definition, regardless of the value of $\lambda$, it has the same variance as $Z_1$, so $\int x^2g_\lambda(x) \dd x = \Var(Z_1)$. Thus the desired inequality is equivalent to 
\[
\int_0^\infty \tilde h(|x|)(g_{\lambda_2(x)} - g_{\lambda_1}(x)) \dd x \geq 0,
\]
with $\tilde h(x) = h(x) + \alpha x^2 + \beta$. Let $x_1, x_2$ be the zeros of $g_{\lambda_2(x)} - g_{\lambda_1}(x)$. Choose $\alpha$ and $\beta$ such that $\tilde h$ has zeros at $x_1$ and $x_2$. Since for $p \geq 3$, $\tilde h(\sqrt{x})$ is convex on $(0,+\infty)$, $\tilde h$ on $(0,+\infty)$ has no other zeros and the sign pattern $+-+$. Thus the integrand is pointwise nonnegative, hence the result.
\end{proof}

Theorem \ref{thm:p>3} is a rather straighforward consequence of the above lemma.

\begin{proof}[Proof of Theorem \ref{thm:p>3}]
By a standard approximation argument of truncating the infinite product if necessary, we can assume that only finitely many $b_j$ are nonzero, say $b_1, \dots, b_n > 0$. By homogeneity, we can also assume that $\sum b_j = 1$. Set $a = 1 + v$ with $v \geq 0$. Then $X$ has the same distribution as $\sum \sqrt{b_j}X_j +  \sqrt{v}Z_0$, where $X_1, X_2, \dots$ are i.i.d. copies of $Z_1$ and $Z_0$ is an independent standard Gaussian. Note that
\[
\sigma^2 = \Var(X) = \sum b_j\Var(X_j) + v\Var(Z_0) = \Var(X_1) + v = 3 + v.
\]

By Lemma \ref{lm:schur},
\[
\E|X|^p \leq \E\left|\sum_{j=1}^n \frac{1}{\sqrt{n}}X_j + \sqrt{v}Z_0\right|^p.
\]
Moreover, also by Lemma \ref{lm:schur}, the right hand side is nondecreasing as $n$ increases. Letting $n \to \infty$ and invoking the central limit theorem gives the upper bound. 

For the lower bound, again thanks to Lemma \ref{lm:schur},
\[
\E|X|^p \geq \E|Z_1 + \sqrt{v}Z_0|^p.
\]
The random variable $\sigma^{-1/2}(Z_1+\sqrt{v}Z_0)$ has the same distribution as $(1+2b)^{-1/2}Z_b$ with $b = \frac{1}{1+v}$, by comparing the characteristic functions. By a direct computation (using the density of $Z_b$),
\[
\E|(1+2b)^{-1/2}Z_b|^p = \frac{2^{p/2}\Gamma\left(\frac{p+1}{2}\right)}{\sqrt{\pi}}\frac{1+pb}{(1+2b)^{p/2}}
\]
and we check that the right hand side as a function of $b$ with $p \geq 3$ fixed is decreasing, which finishes the proof.
\end{proof}

\section{Ultra sub-Gaussianity and related notions}\label{sec:notions}

Here we only consider real-valued random variables.
Recall that after \cite{NO}, a random variable $X$ is called \textbf{ultra sub-Gaussian} if it is symmetric, has all moments finite and the sequence $(a_n)_{n=0}^\infty$ defined by $a_0 = 1$, $a_n = \frac{\E X^{2n}}{\E G^{2n}}$, $n \geq 1$, is log-concave.  Note that then the function
\[
f_X(t) = \E \exp(\sqrt{t}X) = \sum_{n=0}^\infty \frac{\sqrt{t}^{2n}\E X^{2n}}{(2n)!} = \sum_{n=0}^\infty \frac{\E X^{2n}}{(2n-1)!!}\frac{t^n}{2^nn!} = \sum_{n=0}^\infty \frac{a_n}{n!}(t/2)^n
\]
is $C^\infty(\R)$ and the inequalities $a_n^2 \geq a_{n-1}a_{n+1}$, $n \geq 1$, are equivalent to
\begin{equation}\label{eq:f_X-at0}
\big(f_X^{(n)}(0)\big)^2 \geq f_X^{(n-1)}(0)f_X^{(n+1)}(0), \qquad n \geq 1.
\end{equation}
After \cite{Gur}, an entire function $f\colon \C\to\C$ is called \textbf{strongly log-concave} if for every $n \geq 0$, $f^{(n)}$ is either zero or $\log f^{(n)}$ is concave on $\R_+$. Thus we shall say that a random variable is \textbf{strongly log-concave} if it is symmetric, $z \mapsto \E\exp(zX)$ defines an entire function and $f_X^{(n)}$ is log-concave for every $n \geq 0$, where $f_X(t) = \E\exp(\sqrt{t}X)$. Then, in particular, inequalities \eqref{eq:f_X-at0} hold and thus $X$ is also ultra sub-Gaussian. Remarkably, the converse is also true by the following result of Gurvits (see Proposition 4.2(ii) in \cite{Gur}).

\begin{theorem}[Gurvits, \cite{Gur}]\label{thm:Gur}
If a nonnegative sequence $(a_n)_{n=0}^\infty$ is log-concave, then the power series $f(t) = \sum_{n=0}^\infty \frac{a_n}{n!}t^n$ defines a $C^\infty$ function which is log-concave on $\R_+$.
\end{theorem}

Gurvits obtains this as a by-product of his characterisation of \emph{shifts} preserving log-concavity. He also indicates a very different, short argument (see \cite{Gur, Gur2}): by Shephard's theorem (see \cite{Sh}), if a finite sequence $(a_n)_{n=0}^{n_0}$ is log-concave, then there are simplices $K_1, K_2$ in $\R^N$ such that $p_N(t) = \sum_{n=0}^N \binom{N}{n}a_nt^n$ is of the form $\text{vol}_N(K_1+tK_2)$. By the Brunn-Minkowski inequality, $p_N(t)^{1/N}$ is concave, which gives $(p_N(t)')^2 \geq \frac{N}{N-1}p_N(t)\cdot p_N(t)''$. Fix $s > 0$, let $t = s/N$ with $N \to \infty$ to get $(p'(s))^2 \geq p'(s)p''(s)$ with $p(s) = \sum_{n=0}^{n_0} \frac{a_n}{n!}s^n$, allowing to conclude the result. We present a third, rather elementary proof.

\begin{proof}[Proof of Theorem \ref{thm:Gur}]
Thanks to log-concavity, $(a_n)$ is majorised by a geometric sequence, so the radius of convergence of $\sum \frac{a_n}{n!}t^n$ is $\infty$, hence $f$ is $C^\infty$ on $\R$.

We have $f^{(k)}(t) = \sum_{n \geq 0} \frac{a_{n+k}}{n!} t^n$. The log concavity of $f$ is equivalent to $(f')^2 \geq f'' f$, that is
\[
	\left( \sum_{n \geq 0} \frac{a_{n+1}}{n!} t^n \right)^2 - \left( \sum_{n \geq 0} \frac{a_{n}}{n!} t^n \right) \left( \sum_{n \geq 0} \frac{a_{n+2}}{n!} t^n \right) \geq 0.
\]  
The left hand side is a power series with coefficient in front of $t^n$ equal to
\[
	\sum_{k+l=n} \frac{a_{k+1}}{k!} \cdot \frac{a_{l+1}}{l!} - \sum_{k+l=n} \frac{a_{k}}{k!} \cdot \frac{a_{l+2}}{l!} .
\]
It is enough to show that this expression is non-negative. Multiplying both sides by $n!$, we see that it is enough to prove the inequality
\[
	\sum {n \choose k } a_{k+1} a_{n-k+1} - \sum {n \choose k } a_{k} a_{n-k+2} \geq 0
\]
with the usual convention that ${n \choose k} =0$ if $k<0$ or $k>n$. The sums are over $\Z$. Shifting the summation index in the first sum we get an equivalent form
\[
	\sum \left({n \choose k-1} - {n \choose k} \right) a_k a_{n-k+2} \geq 0.
\]
For $n=0$ this is just $a_1^2 \geq a_0 a_2$ and for $n=1$, it is $2a_1 a_2 \geq a_0 a_3 + a_1 a_2$, which is just $a_1 a_2 \geq a_0 a_3$. By log-concavity, both hold. Thus in what follows we can assume that $n \geq 2$.

Let $b_k = a_k a_{n-k+2}$ (we set $a_k=0$ for $k<0$). Our goal is to prove
\[
	\sum \left({n \choose k-1} - {n \choose k} \right) b_k \geq 0.
\]
Note that $b_k = b_{n-k+2}$. From log concavity we know that the sequence $b_k$ is nondecreasing, $b_{k-1} \leq b_{k}$ for $k \leq \frac{n}{2}+1$.
Since $b_k$ is symmetric about $k=\frac{n}{2}+1$, we can write the desired inequality in an equivalent form
\[
	\left({n \choose \frac{n}{2}} - {n \choose \frac{n}{2}-1} \right) b_{\frac{n}{2}+1} +  \sum_{k < \frac{n}{2}+1} \left(2 {n \choose k-1} - {n \choose k} - {n \choose k-2 } \right) b_k \geq 0
\]
with the first term present only if $n$ is even. To prove this, we set
\[
s_k = 2 {n \choose k-1} - {n \choose k} - {n \choose k-2 }, \qquad k < \frac{n}{2}+1
\]
and additionally when $n$ is even,
\[
s_{n/2+1} = {n \choose \frac{n}{2}} - {n \choose \frac{n}{2}-1}
\]
and use the following elementary lemma (Lemma 6 in \cite{NO}). 
\begin{lemma}[Nayar-Oleszkiewicz, \cite{NO}]\label{lm:NO}
If $b_0 \leq b_1 \leq \ldots \leq b_m$ and real numbers $s_0, s_1, \ldots, s_m$ are such that $\sum_{k=0}^m s_k = 0$ and the sequence $(\sgn(s_k))$ is nondecreasing, then $\sum_{k=0}^m s_k b_k \geq 0$.
\end{lemma}

The lemma is applied with $m = \lfloor n/2+1 \rfloor$ and the sequence $(s_k)$ defined above (which plainly satisfies $\sum_{k=0}^m s_k = 0$). It remains to check that $(sgn(s_k))$ is a nondecreasing sequence. 
For $n \geq 2$ the inequality $s_k \leq 0$ is equivalent to
\[
	-4k^2 +4kn +8k - n^2 - 3n-2 \leq 0.
\] 
The maximum of this parabola is attained for $k=\frac{n}{2}+1$ and thus for $k \leq \frac{n}{2}+1$ this quadratic function is increasing. This shows that $(\sgn(s_k))$ is non-decreasing, which finishes the proof.
\end{proof}

This gives that the classes of strongly log-concave and ultra sub-Gaussian random variables are the same. As observed earlier in Theorem \ref{thm:type-L-USG}, type $\sL$ random variables are ultra sub-Gaussian. The converse is false: it is enough to consider $X$ with $3$ atoms: $\p{X=0} = 1-\beta$, $\p{X = \pm 1} = \beta/2$. Then $\E\exp(zX) = (1-\beta)+ \beta \cosh(z)$. The equation $(1-\beta)+ \beta \cosh(z)=0$ has only purely imaginary zeroes only if $\beta>1/2$. On the other hand, this random variables is ultra sub-Gaussian for any $\beta \in (0,1)$. Another examples are random variables with densities proportional to $e^{-|x|^\alpha}$: they are ultra sub-Gaussian for every $\alpha \geq 2$ (see Corollary 2 in \cite{NO}), whereas they are of type $\sL$ only for integral $\alpha$ (see Section \ref{sec:examples}). This discussion is summarised in Figure \ref{fig:dom}.

\usetikzlibrary{shapes.geometric, arrows}
\tikzstyle{box} = [rectangle, rounded corners, minimum width=2cm, minimum height=0.5cm,text centered, draw=black]
\tikzstyle{arrow} = [->,>=stealth]

\begin{figure}[ht]

\begin{center}
\begin{tikzpicture}[node distance=2cm]
\node (USG) [box] {USG};
\node (FLC) [box, right of=USG, xshift=2cm] {SLC};
\node (L) [box, below of=USG, yshift=0.5cm, xshift=2cm] {type $\sL$};
\draw [arrow] (USG) -- node {/} (L);
\draw [arrow]  (L.160) -- (USG.-108);
\draw [arrow]  (L.20) -- (FLC.-90);
\draw [arrow] (FLC) -- node {/} (L);
\draw [arrow] (USG.5) -- (FLC.175);
\draw [arrow] (FLC.-175) -- (USG.-5);
\end{tikzpicture}
\end{center}
\caption{Implications between ultra sub-Gaussian (USG), strongly log-concave (SLC) and type $\sL$ random variables.}
\label{fig:dom}

\end{figure}
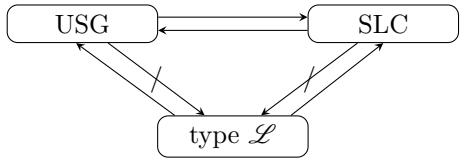

\end{document}